\documentclass[reqno,centertags,11pt,a4paper]{amsart}
\usepackage{amssymb,mathrsfs}
\usepackage{color,umoline}
\usepackage[dvipsnames]{xcolor}
\usepackage{graphicx}
\usepackage{enumitem}
\usepackage[font=small,labelfont=bf]{caption}
\usepackage{tikz, caption, subcaption}
\usepackage{relsize}
\usepackage[mathscr]{eucal}
\usepackage{verbatim}
\usepackage[draft]{todonotes}
\usepackage{hyperref}
\usepackage[dvipsnames]{xcolor}
\usetikzlibrary{patterns}
\tikzstyle{EDR}=[draw=lightgray,line width=0pt,preaction={clip, postaction={pattern=north east lines, pattern color=gray}}]
\tikzstyle{EDR1}=[draw=lightgray,line width=0pt,preaction={clip, postaction={pattern=north west lines, pattern color=gray}}]

\definecolor{mygray}{gray}{0.95}

\definecolor{mypink1}{rgb}{1.2,1.1,0.9}

\definecolor{mypink2}{rgb}{1.0,0.95 ,0.9}

\definecolor{mypink3}{rgb}{1.0,0.6,0.7}

\numberwithin{equation}{section}

\newtheorem{theorem}{Theorem}[section]
\newtheorem{definition}[theorem]{Definition}
\newtheorem{lemma}[theorem]{Lemma}

\newtheorem{proposition}[theorem]{Proposition}

\newtheorem{remark}[theorem]{Remark}

\numberwithin{equation}{section}

\newcommand{\R}{\mathbb R}

\newcommand{\dif}{\mathrm{d}}
\newcommand{\beq}{\begin{equation}}
	\newcommand{\eeq}{\end{equation}}
\newcommand{\beqq}{\begin{equation*}}
	\newcommand{\eeqq}{\end{equation*}}
\newcommand{\ben}{\begin{eqnarray}}
	\newcommand{\een}{\end{eqnarray}}
\newcommand{\beno}{\begin{eqnarray*}}
	\newcommand{\eeno}{\end{eqnarray*}}

\begin{document}
		\title[H\"ormander oscillatory integral operators]
	{H\"ormander oscillatory integral operators: a revisit}

	\author[C. Gao]{Chuanwei Gao}
\address{	School of Mathematical Sciences, Capital Normal University, Beijing 100048, China}
\email{cwgao@cnu.edu.cn}

	\author[Z. Gao]{Zhong Gao}
\address{Institute of Applied Physics and Computational Mathematics, P.O. Box 8009, Beijing 100088, P.R. China }
\email{gaozhong18@gscaep.ac.cn}

	\author[C. Miao]{Changxing Miao}

\address{Institute of Applied Physics and Computational Mathematics, P.O. Box 8009, Beijing 100088, P.R. China }

\email{miao\_{}changxing@iapcm.ac.cn}
	



	\begin{abstract}
		In this paper, we present new proofs for both the sharp $L^p$ estimate and the decoupling theorem for the H\"ormander  oscillatory integral operator. The sharp $L^p$ estimate was previously obtained by Stein\;\cite{stein1} and Bourgain-Guth \cite{BG} via the $TT^\ast$ and multilinear methods, respectively. We provide a  unified  proof  based on the bilinear method for both odd and even dimensions. The strategy is inspired by  Barron's work \cite{Bar} on the restriction problem. The decoupling theorem for the H\"ormander oscillatory integral operator can be obtained by the approach in \cite{BHS}, where the key observation can be roughly formulated as follows: in a physical space of sufficiently small scale, the variable setting can be essentially viewed as translation-invariant. In contrast, we reprove the decoupling theorem for the H\"ormander oscillatory integral operator through the Pramanik-Seeger approximation approach \cite{PS}. Both proofs rely  on a scale-dependent induction argument,  which can be used to deal with perturbation terms in the phase function.

	\end{abstract}
	
	\maketitle
	
	\section{Introduction}
	Let $n \geq 2$, $a \in C_c^{\infty}(\R^n \times \R^{n-1})$ be non-negative and supported in $B^n_1(0) \times B^{n-1}_1(0)$ and $\phi \colon B^n_1(0) \times B^{n-1}_1(0) \to \R$ be a smooth function. For any $\lambda\geq 1$, define 
	\begin{equation}\label{eq:00}
		T^{\lambda}f(x) :=  \int_{B^{n-1}_1(0)} e^{2 \pi i \phi^{\lambda}(x, \xi)}a^{\lambda}(x,\xi) f(\xi)\,d \xi,
	\end{equation}
	where  $f \colon B^{n-1}_1(0) \to \mathbb{C}$ and
	\begin{equation}\label{eq:00a}
		a^{\lambda}(x, \xi) := a(x/\lambda, \xi),\  \phi^{\lambda}(x,\xi) :=\lambda\phi(x/\lambda,\xi).
	\end{equation}
	We say that the operator $T^\lambda$ is a H\"ormander oscillaroty integral operator if $\phi$ satisfies the following Carleson-Sj\"olin conditions:
	\begin{itemize}
		\item[(H1)] $\mathrm{rank}\, \partial_{x'\xi}^2 \phi(x,\xi) = n-1$ for all $(x,\xi) \in B^n _1(0)\times B^{n-1}_1(0)$ and $x=(x',x_n)$;
		\item[(H2)] For each $x_0\in{\rm supp\,}_x a$, the hypersurface $$\{\partial_x\phi(x_0,\xi): \xi \in {\rm supp}\; a(x_0,\cdot)\}$$
		has non-vanishing Gaussian curvature.
		
	\end{itemize}
	A typical example for the H\"ormander oscillatory  integral operators is the following extension operator $E$ defined by
	\beq\label{eq:02a}
	Ef(x):=\int_{B_1^{n-1}(0)} e^{2\pi i(x'\cdot \xi+x_n\psi(\xi))}f(\xi)d\xi,
	\eeq
	with
	\beq\label{eq:02z}
	{\rm rank}\Big(\frac{\partial^2 \psi}{\partial \xi_i\partial \xi_j}\Big)_{(n-1)\times (n-1)}=n-1,
	\eeq
it is straightforward to verify that the phase function $\phi(x,\xi):=x'\cdot\xi+x_n\psi(\xi)$ satisfies the conditions (H1) and (H2).
	For the H\"ormander oscillatory integral operators, we revisit the following two important problems: the sharp $L^p$ estimate and the decoupling inequality.
	
	{\bf \noindent Sharp $L^p$ estimate.}
	H\"ormander \cite{Hor} conjectured that  if $\phi$ satisfies conditions (H1),(H2), then
	\beq\label{eq:01a}
	\|T^\lambda f\|_{L^p(\R^n)}\lesssim \|f\|_{L^p(B_1^{n-1}(0))}
	\eeq
	for $p>\frac{2n}{n-1}$, and he proved this conjecture for $n=2$. For the higher dimensional cases, Stein \cite{stein1} proved \eqref{eq:01a}  for $p\geq \frac{2(n+1)}{n-1}$. Later, Bourgain \cite{Bourgain1} disproved H\"ormander's conjecture, and showed that Stein's result is sharp in the odd dimensions. For the even dimensions,  up to an endpoint, Bourgain and Guth \cite{BG} established the sharp result. In summary, we  may state  the results as follows.
	\begin{theorem}[\cite{BG},\cite{stein1}]\label{theo1}
		Let $n\geq 3$ and  $T^\lambda$ be a H\"ormander oscillatory integral  operator as in $\eqref{eq:00}$. For all $\varepsilon>0,\lambda\geq 1$,
		\beq
		\|T^\lambda f\|_{L^p(\R^n)}\lesssim_{\varepsilon,\phi,a} \lambda^{\varepsilon}\|f\|_{L^p(B^{n-1}_1(0))}
		\eeq
		holds whenever
		\begin{equation}\label{eq:mainaa}p\geq \left\{\begin{aligned}
				&\tfrac{2(n+1)}{n-1}\quad \text{\rm for $n$ odd},\\
				&\tfrac{2(n+2)}{n}\quad \text{\rm for $n$ even}.
			\end{aligned}\right.\end{equation}
	\end{theorem}
	
	Stein's proof is based on the $TT^{\ast}$ method and gives the range $p\geq \frac{2(n+1)}{n-1}$ in all dimensions. However, this result is not sharp in even dimensions. Bourgain-Guth \cite{BG} resolved the even-dimensional cases up to the endpoints using the Broad-Narrow approach. Bourgain-Guth's method can also be applied to the odd-dimensional cases, see Guth-Hickman-Iliopoulou\cite{GHI} for details. We give another proof based on the bilinear approach. We take the extension operator as a model case to illustrate how one can derive  the linear estimate for the oscillatory integral operator from its bilinear counterpart. To this end, we first recall a sharp bilinear restriction theorem of Lee\cite{Lee}.
	
	\begin{theorem}[\cite{Lee}\label{theo2}]
		Suppose that $\xi\in B_1^{n-1}(0)$ and the \emph{Hessian matrix} of $\phi$ is nondegenerate, i.e.
		\beqq
		{\rm det}\mathcal H\phi(\xi)\neq 0.
		\eeqq
		Additionally, let $V_1,V_2$ be two sufficiently small balls contained in $B^{n-1}_1(0)$, and suppose that for all $\xi'\in V_1, \xi''\in V_2$ and $\xi_i \in V_i, i=1,2$,
		\beq\label{co1}
		\big|\langle (\mathcal{H}\phi)^{-1}(\xi_i)(\nabla\phi(\xi')-\nabla\phi(\xi'')), \nabla\phi(\xi')-\nabla\phi(\xi'')\rangle\big|
		\geq c>0,\eeq
		then
		\beq
		\||Ef_1Ef_2|^\frac12\|_{L^p(\R^{n})}\leq R^\varepsilon \|f_1\|_{L^2}^\frac12\|f_2\|_{L^2}^\frac12,
		\eeq
		for $p\geq \frac{2(n+2)}{n}$.
	\end{theorem}
	
	To derive the linear estimate from Theorem \ref{theo2}, an important step is to identify  the exceptional set where the condition \eqref{co1} fails. When the Hessian of $\phi$ has  eigenvalues of  the same sign, the separation of $V_1$ and $V_2$ is sufficient to guarantee the condition \eqref{co1}. However, this fact does not hold true when the Hessian of $\phi$ has eigenvalues with different signs. For example, when $n=2$,  if $\phi_{\rm hyp}(\xi)=\xi_1\xi_2$, the exceptional set may be contained in a small neighborhood of coordinates. For the general phase $\phi_{\rm M}$ which can be viewed as a small perturbation of $\phi_{\rm hyp}$, identifying the exceptional set is a bit tricky. There are a number of papers by Buschenhenke-M\"uler-Vargas\cite{BMV1,BMV2, BMV3, BMV}  which are dedicated  to the study of the restriction estimate associated with the phase $\phi_{\rm M}$. However, it's still murky to find the exceptional set for the phase $\phi_{\rm M}$ in the higher dimensional cases. To circumvent this issue, inspired by the work of \cite{GLMX}, we consider a class of scale-dependent phase functions.  Their exceptional set can be connected  with the  quadratic cases of which  the exceptional set is clear.

	{\bf \noindent Decoupling theorem.}   Assume $\{{\theta}\}$ is a collection of finitely overlapping balls in $\R^{n-1}$ of radius $R^{-1/2}$ which form a cover of $B_1^{n-1}(0)$.
	Define $$ f_{\theta}:=f\kappa_\theta;\; \sum_{\theta}\kappa_\theta=1,\;\forall \xi \in B_1^{n-1}(0)$$
	where $\{\kappa_\theta\} $ is a family of smooth functions which are subjecting to $\{\theta\}$.
	Correspondingly, we decompose $T^\lambda f$ into
	\beqq
	T^\lambda f:=\sum_\theta  T^\lambda f_\theta.
	\eeqq
	We  have the following decoupling theorem for the H\"ormander oscillatory  operator.
	\begin{theorem}\label{theo31}
		Let $T^\lambda$ be a H\"ormander oscillatory integral  operator as in $\eqref{eq:00}$.  If $p\geq \frac{2(n+1)}{n-1}$, then
		\beq
		\big\|\sum_{\theta}T^\lambda f_{\theta}\big\|_{L^p(B_R^n(x_0))}\leq C_\varepsilon R^{\frac {n-1}2-\frac np+\varepsilon}\big(\sum_\theta \|T^\lambda f_\theta\|_{L^p(w_{B_R^n(x_0)})}^p\big)^{\frac{1}{p}}+{\rm RapDec}(R)\|f\|_{2},
		\eeq
		where $w_{B_R^n(x_0)}$ is a non-negative  weight function adapted to the ball $B_R^n(x_0)$ such that
		$$ w_{B^{n}_{R}(x_0)}(x)\lesssim (1+R^{-1}|x-x_0|)^{-L},$$
		for some large constant $L\in \mathbb{N}$.
	\end{theorem}
	The decoupling theorem for the extension operator was established by Bourgain-Demeter \cite{BD1}. When the phase function satisfies the cinematic curvature condition, the associated variable version of the decoupling theorem was established by Beltran-Hickman-Sogge \cite{BHS}(see also \cite{ILX}). Their method can also be applied to the H\"ormander oscillatory integral operator. A key observation in \cite{BHS} can be roughly formulated as follows: At the small scale of physical space, the variable setting is essentially translation invariant. Hence, the decoupling theorem for the flat version can be brought into play directly at the level of a small scale of physical space.
	
	We present an alternative proof of the theorem \ref{theo31} based on the Pramanik-Seeger's approach \cite{PS}. To be more precise, we first conduct a localization procedure in frequency space. In this setting, the key is to effectively control the error term so that we can directly use the techniques in the translation-invariant setting.

	The rest  is organized as follows: In Section  \ref{pre}, we present some preliminaries which are useful for the proof of the main theorem. In Section \ref{hor}, we prove the sharp $L^p$ estimate for H\"ormander oscillatory integral operator. In Section \ref{dec}, we provide the proof of the decoupling theorem for H\"ormander oscillatory integral operators.
	
	{\bf\noindent  Notations.} For nonnegative quantities $X$ and $Y$, we will write $X\lesssim Y$ to denote the inequality $X\leq C Y$ for some $C>0$. If $X\lesssim Y\lesssim X$, we will write $X\sim Y$. We write $x\mapsto y$ to mean that we replace $x$ by $y$. Dependence of implicit constants on the spatial dimensions or integral exponents such as $p$ will be suppressed; dependence on additional parameters will be indicated by subscripts. For example, $X\lesssim_u Y$ indicates $X\leq CY$ for some $C=C(u)$. We write $A(R)\leq {\rm RapDec}(R) B$ to mean that for any power $\beta\in\mathbb N$, there is a constant $C_\beta$ such that
	\beqq
	| A(R)|\leq C_\beta R^{-\beta}B \,\,\quad\text{for all}\,\; R\geq 1.
	\eeqq
	We will also often abbreviate $\|f\|_{L_x^r(\R^n)}$ to $\|f\|_{L^r}$.  For $1\leq r\leq\infty$,
	we use $r'$ to denote the dual exponent to $r$ such that $\tfrac{1}{r}+\tfrac{1}{r'}=1$. Throughout the paper, $\chi_E$ is the characteristic function of the set $E$.
	We usually denote by $B_r^n(a)$ a ball in $\R^n$ with center $a$ and radius $r$. We will also denote by $B_R^n$ a ball of radius $R$ and arbitrary center in $\R^n$.
	For a function $\varphi\in C^\infty(\mathbb R^n)$ and $r>0$, we define $\varphi_r(x)=r^{-n}\varphi(x/r)$.
	
	We define the Fourier transform on $\mathbb{R}^n$ by
	\begin{equation*}
		\aligned \hat{f}(\xi):= \int_{\mathbb{R}^n}e^{-2\pi  ix\cdot \xi}f(x)\,\dif x:=\mathcal{F}f(\xi).
		\endaligned
	\end{equation*}
	and the inverse Fourier transform by
	\beqq
	\check{g}(x):=\int_{\R^n} e^{2\pi ix\cdot \xi}g(\xi)\dif \xi:=(\mathcal{F}^{-1}g)(x).
	\eeqq

	\section{Basic reductions}\label{pre}
Let $m \in \mathbb{N}$ be the number of the positive eigenvalues of the hypersurfaces $\{\partial_x\phi(x,\xi): (x,\xi)\in B_1^n(0)\times B_1^{n-1}(0)\}$. Instead of dealing with the phase $\phi$ directly, we actually reduce it to a special class of functions. Let $M$ be a diagonal matrix with its entries being either  $-1$ or $1$ in the diagonal. Analytically, we can express  $M$ as follows
	\beqq
	M=-I_{n-1-m}\oplus I_m,
	\eeqq
	for some $1\leq m\leq \lfloor\frac{n-1}{2}\rfloor$\footnote{$\lfloor x\rfloor$ denotes the greatest integer  less than or equal to $x$}.
	\begin{definition}\label{def00}
		Let $K \geq 1$ and the $\phi_K:B_1^n(0)\times B_1^{n-1}(0)\longmapsto  \R$ with
		\beq\label{ph1}
		\phi_K (x,\xi)=x'\cdot \xi +x_n\langle  M\xi,\xi\rangle +{\rm E}_K(x,\xi).
		\eeq
		We say that the phase function $\phi_K(x,\xi)$ is \emph{asymptotically flat} if
		\beq \label{co123}
		|\partial^\alpha_x  \partial^\beta_\xi {\rm E}_K(x,\xi)|\leq C_{\alpha,\beta}K^{-2},\quad (\alpha,\beta)\in \mathbb{N}^{n}\times \mathbb{N}^{n-1},|\alpha|\leq N_{\rm ph}, |\beta|\leq N_{\rm ph},
		\eeq
		where $N_{\rm ph}\in \mathbb{N}$ is a large integer and $C_{\alpha,\beta}>0$ is a constant depending on $\alpha, \beta$ but not on $K$.
	\end{definition}

\begin{remark}
The phase function $\phi_K(x,\xi)$ in \eqref{ph1} depends on the scale of ambient space. 
We can exploit the properties  in \eqref{ph1} and \eqref{co123} in the process of induction on scales argument since the balls shrink after the parabolic rescaling transformation.
\end{remark}

	In the following part, let $R\gg 1, K=K_0R^{\delta}$ for some constants  $K_0>0,\delta>0$ to be chosen later, define the operator $T_K^\lambda$ as follows
	\beq\label{phiK}
	T_K^\lambda  f(x):=\int_{B_1^{n-1}(0)}e^{i \phi_K^\lambda(x,\xi)}\mathfrak{a}^\lambda(x,\xi)f(\xi)d\xi,
	\eeq
where $\phi^\lambda_K,\ \mathfrak{a}^\lambda$ are defined in the same way with $\eqref{eq:00a}$  and $\mathfrak a$ is a smooth cut function in $\R^{n}\times\R^{n-1}$
satisfying:
 ${\rm supp}\;\mathfrak{a}(x,\xi)\subset B_1^{n}(0)\times B_1^{n-1}(0)$ and
		\beq\label{eq:c13}
		|\partial_x^\alpha \partial_\xi^\beta \mathfrak{a}(x,\xi)|\leq \bar{C}_{\alpha,\beta},\; (\alpha,\beta)\in \mathbb{N}^{n}\times \mathbb{N}^{n-1}, |\alpha|\leq N_{\rm am}, |\beta|\leq N_{\rm am},
		\eeq
		for an appropriate large constant $N_{\rm am}\in \mathbb{N}$.
	\begin{lemma}\label{la11}
		Let $T^\lambda$ be an H\"ormander oscillatory integral operator defined by \eqref{eq:00} and $\delta\ll \varepsilon$,  then there exists a function $\phi_K$ which is asymptotically flat   and an input function  $\tilde f$ defined by
		\beq\label{eq:c12}
		\tilde{f}(\xi):=K^{-(n-1)}f(\bar{\xi}+K^{-1}\xi),\;\; \text{for some}\;\bar{\xi}\in B_1^{n-1}(0),
		\eeq
		such that
		\beq
		\|T^\lambda f\|_{L^p(B_R^n(0))}^p\lesssim_{\phi,\varepsilon} R^\varepsilon\sum_{B_{\tilde R}^n \subset \Box_{R} } \|T^{\tilde \lambda}_{\tilde K} \tilde f\|_{L^p(B_{\tilde R}^{n})}^p,
		\eeq
		where $\tilde R:=R/K^2$, $\tilde{K}:=K_0\tilde{R}^{\delta}$, $\Box_R$ is  a rectangular box of dimensions  $R/K\times \cdots\times R/K\times R/K^2$ and $\{B_{\tilde R}^n\}$ is a finitely overlapping partition of $\Box_R$.
	\end{lemma}

	\begin{proof}
Without loss of generality, we may assume
	\beqq
	|\partial^\alpha_x \partial^\beta_\xi \phi(x,\xi)|\leq C_{\alpha,\beta} K^{-2},\quad 2\leq  | \alpha|\leq N_{\rm ph},\,|\beta|\leq N_{\rm ph}.
	\eeqq
	Otherwise, we may replace $\phi(x,\xi)$ by $\phi(x/A,\xi)$ where $A$ is a sufficiently large constant depending on $K$. It should be noted that the support of $a(x/A,\xi)$ may be not contained in $B_1^{n}(0)\times B_1^{n-1}(0)$, but this  can be fixed by a partition of unity argument.

Covering $B_1^{n-1}(0)$ by a collection of balls $\{\tau\}$ of radius $K^{-1}$ and define
		$f_\tau:=f\chi_\tau$.
		By the triangle inequality, we have
		\beqq
		\|T^\lambda f\|_{L^p(B_R^n(0))}\leq \sum_{\tau}\|T^\lambda f_\tau\|_{L^p(B_R^n(0))}.
		\eeqq
		Thus, there exits a $\tau_0$ such that
		\beqq
		\sum_\tau \|T^\lambda f_\tau\|_{L^p(B_R^n(0))}\lesssim K^{n-1} \|T^\lambda f_{\tau_0}\|_{L^p(B_R^n(0))}.
		\eeqq
		Without loss of generality, we may assume $\xi_{\tau_0}$ is the center of $\tau_0$ and
		\beqq
		\partial_x^\alpha \phi^\lambda (x,\xi_{\tau_0})=0,\quad \partial_\xi^\beta  \phi^\lambda (0,\xi)=0, \quad \alpha\in \mathbb{N}^n,\,\beta\in \mathbb{N}^{n-1}.
		\eeqq
		Otherwise, we take $\phi^\lambda$ to be
		\beqq
		\phi^\lambda(x,\xi)+\phi^\lambda(0,\xi_{\tau_0})-\phi^\lambda(0,\xi)-\phi^\lambda(x,\xi_{\tau_0}).
		\eeqq
		By an affine transformation in $x$, we may also assume  the unit normal vector of the hypersurface $\{\partial_x\phi^\lambda (0,\xi):\xi \in \tau_0\}$ at $\xi=\xi_{\tau_0}$ equals $(0,\cdots,0,1)$ and $\partial_{\xi\xi}\partial_{x_n}\phi^\lambda(0,\xi_{\tau_0})=M$. Thus we have
		$${\rm rank}\partial_{x'}\partial_\xi \phi^\lambda (x,\xi)=n-1,\quad (x,\xi)\in B^n_1(0)\times B^{n-1}_{K^{-1}}(\xi_{\tau_0}).$$
		Then by the inverse function theorem, there exists a function $\Phi^\lambda(x',x_n)$ such that
		\beqq
		\partial_\xi \phi^\lambda(\Phi^\lambda(x',x_n),x_n,\xi_{\tau_0})=x'.
		\eeqq
		By a change of variables in $\xi$
		\beqq
		\xi \longmapsto \xi+\xi_{\tau_0},
		\eeqq
		 and  Taylor's formula,  we have
		\beq
		\begin{aligned}
			\phi^\lambda(x, \xi+\xi_{\tau_0})&=\partial_\xi\phi^\lambda(x,\xi_{\tau_0})\cdot \xi+\frac{1}{2}\langle \partial_{\xi\xi}^2\phi^\lambda(x,\xi_{\tau_0})\xi,\xi\rangle \\&+3\sum_{|\beta|=3}\frac{\xi^\beta}{\beta!}\int_0^1 (1-t)^2 \partial_{\xi}^\beta\phi^\lambda(x,\xi_{\tau_0} +t\xi)dt.
		\end{aligned}
		\eeq
		We make another change of variables in $x$
		\beqq
		x'\longmapsto  \Phi^\lambda(x',x_n),\  x_n\longmapsto x_n,
		\eeqq
		such that in the new coordinates,  the phase becomes
		\beqq
		\begin{aligned}
			\langle x',\xi  \rangle &+\frac{1}{2}\langle \partial_{\xi\xi}^2\phi^\lambda(\Phi^\lambda(x',x_n),x_n,\xi_{\tau_0})\xi,\xi\rangle\\&+3\sum_{|\beta|=3}\frac{\xi^\beta}{\beta!}\int_0^1 (1-t)^2 \partial_{\xi}^\beta \phi^\lambda(\Phi^\lambda(x',x_n),x_n,\xi_{\tau_0} +t\xi)dt.
		\end{aligned}
		\eeqq
		Let $\mathcal A_\phi^\lambda(x,\xi_{\tau_0}):=\lambda\mathcal A_\phi(x/\lambda,\xi):=\partial_{\xi\xi}^2\phi^\lambda(\Phi^\lambda(x',x_n),x_n,\xi_{\tau_0})$,  then  a Taylor expansion in $x$ yields
		\beqq
		\begin{aligned}
			\langle x',\xi  \rangle &+\frac{1}{2}x_n\big(\partial_{x_n}\langle\mathcal A_{\phi}^\lambda (x,\xi_{\tau_0})\xi,\xi\rangle\big)\big\vert_{x=0}+\frac{1}{2}x'\cdot\big(\partial_{x'}\langle\mathcal A_\phi^\lambda(x,\xi_{\tau_0})\xi,\xi\rangle\big)\big\vert_{x=0}\\&
			+2\sum_{|\alpha|=2}\frac{x^\alpha}{\alpha!}\int_0^1 (1-t) \partial_{z}^\alpha\langle \mathcal A_\phi^\lambda(z,\xi_{\tau_0}) \xi,\xi\rangle\big\vert_{z=tx} dt
			\\&+3\sum_{|\beta|=3}\frac{\xi^\beta}{\beta!}\int_0^1 (1-t)^2 \partial_{\xi}^\beta \phi^\lambda(\Phi^\lambda(x',x_n),\xi_{\tau_0} +t\xi)dt.
		\end{aligned}
		\eeqq
		 We make a further diffeomorphic change of variables in $\xi\longmapsto \rho(\xi)$ such that in the new coordinates, $
		\langle x',\xi  \rangle+\frac{1}{2}x'\cdot \big(\partial_{x'}\langle\mathcal A_\phi^\lambda(x,\xi_{\tau_0})\xi,\xi\rangle\big)\big\vert_{x=0} $
		becomes $\langle x',\xi\rangle$. It is obvious that $\rho(0)=0$, thus a further Taylor expansion in $\xi$ for
$\frac12x_n(\partial_{x_n}\langle \mathcal A_\phi^\lambda(x,\xi_{\tau_0})\rho(\xi),\rho(\xi)\rangle)\vert_{x=0}$, up to an affine transfromation in $\xi$,
we have
        \beqq
\frac12x_n(\partial_{x_n}\langle \mathcal A_\phi^\lambda(x,\xi_{\tau_0})\rho(\xi),\rho(\xi)\rangle)\vert_{x=0}=\frac{1}{2}x_n\langle M\xi,\xi\rangle+x_nr(\xi),
        \eeqq
        where $r(\xi)=O(|\xi|^3)$. Define
		\beqq
		\begin{aligned}
			{\rm E}(x,\xi):=2\sum_{|\alpha|=2}\frac{x^\alpha}{\alpha!}&\int_0^1 (1-t) \partial_{z}^\alpha\langle \mathcal A_\phi(z,\xi_{\tau_0}) \rho(\xi),\rho(\xi )\rangle\big\vert_{z=tx}dt
			\\&+3\sum_{|\beta|=3}\frac{(\rho(\xi))^\beta}{\beta!}\int_0^1 (1-t)^2 \partial_{\xi}^\beta\phi(\Phi(x',x_n),\xi_{\tau_0} +t\rho(\xi))dt+x_nr(\xi).
		\end{aligned}
		\eeqq
	Correspondingly,  the phase function becomes
\beqq
\langle x',\xi\rangle+\frac12x_n\langle M\xi,\xi\rangle+{\rm E}^\lambda(x,\xi),
\eeqq
where ${\rm E}^\lambda(x,\xi):=\lambda {\rm E}(x/\lambda,\xi)$.
Define $\tilde{\lambda}:=\lambda/K^2, \tilde R:=R/K^2, \tilde K:=K_0(\tilde R)^\delta$.
We perform a parabolic rescaling
		\beqq
		\xi\longmapsto K^{-1}\xi, \; x'\longmapsto Kx', \; x_n\longmapsto  K^2 x_n.
		\eeqq
	The phase function  becomes
		\beqq
		\phi_{\tilde K}^{\tilde \lambda}(x,\xi):=\langle x',\xi \rangle+x_n\langle M\xi,\xi\rangle+{\rm E}_{\tilde K}^{\tilde \lambda}(x,\xi),
		\eeqq
		where
		$${\rm E}_{\tilde K}(x,\xi):=K^2{\rm E}(K^{-1}x',x_n, K^{-1}\xi),$$
and ${\rm E}_{\tilde K}^{\tilde \lambda}(x,\xi):=\tilde \lambda{\rm E}_{\tilde K}(x/\tilde \lambda,\xi)$.  Finally, we have 
\begin{align}T_{\tilde K}^{\tilde \lambda}\tilde f(x)=\int_{B_1^{n-1}(0)}e^{2\pi i\phi_{\tilde K}^{\tilde\lambda}(x,\xi)}\mathfrak a^{\tilde \lambda}(x,\xi)\tilde f(\xi)d\xi.\end{align}
		Note our assumption on $\tilde{\phi}$, it is straightforward to verify that ${\rm E}_{\tilde K}(x,\xi)$
		satisfies the condition
		\beq \label{co12}
		|\partial^\alpha_x  \partial^\beta_\xi {\rm E}_{\tilde K}(x,\xi)|\leq C_{\alpha,\beta}{\tilde K}^{-2},\quad (\alpha,\beta)\in \mathbb{N}^{n}\times \mathbb{N}^{n-1},|\alpha|\leq N_{\rm ph },  |\beta|\leq N_{\rm ph}.
		\eeq
		Thus $\phi_{\tilde K}$ is asymptotically flat. Under the new coordinates, the phase function becomes $\phi^{\tilde \lambda}_{\tilde K}$. By tracking the change of variables of $\xi$ and $x$, it is easy to see the ball $B_R^n$ is transformed into another region  which  is contained in a box $\Box_{R}$ of dimensions $R/K\times \cdots \times R/K\times R/K^2$ and by choosing $K_0$ sufficiently large, the condition $\eqref{co123},\eqref{eq:c13}$ can be ensured.
	\end{proof}

	\section{proof of the sharp $L^p$ estimate}	\label{hor}
	{\bf \noindent Reduction:}
	To prove Theorem \ref{theo1}, it suffices to show for each $1\leq R\leq \lambda$,
	\beq \label{re00}
	\|T^\lambda f\|_{L^p(B_R^n(0))}\leq C_\varepsilon R^\varepsilon\|f\|_{L^p(B^{n-1}_1(0))}
	\eeq
    under the assumption $\eqref{eq:mainaa}$.
	The dependence of the implicit constant on $n,p,\phi$ is compressed. By Lemma \ref{la11}, it is reduced to showing  for each $1 \leq R \leq \lambda$,
	\beq\label{eq:123}
	\|T^\lambda_K f\|_{L^p(B_R^n)}\leq C_\varepsilon R^\varepsilon \|f\|_{L^p(B^{n-1}_1(0))},
	\eeq
for all $T_K^\lambda$ as in $\eqref{phiK}$.
	Indeed, by Lemma \ref{la11}, we have
	\beqq
	\|T^\lambda f\|_{L^p(B_R^n(0))}\lesssim_{\phi,\varepsilon} R^\varepsilon\sum_{B_{\tilde R}^n \subset \Box_{R} } \|T^{\tilde \lambda}_{\tilde K} \tilde f\|_{L^p(B_{\tilde R}^{n})}^p,
	\eeqq
where \[T^{\tilde\lambda}_{\tilde K}\tilde f(x)=\int_{B_1^{n-1}(0)}e^{2\pi i\phi_{\tilde K}^{\tilde\lambda}(x,\xi)}\mathfrak a^{\tilde\lambda}(x,\xi)\tilde f(\xi)d\xi\]
and $\tilde f$ is defined by $\eqref{eq:c12}$.
	Note that $K=K_0R^\delta$, there exists a $\bar{B}_{\tilde R}^n\subset \Box_R$ such that
	\beqq
	\sum_{B_{\tilde R}^n \subset \Box_{R} } \|T^{\tilde \lambda}_{\tilde K} \tilde f\|_{L^p(B_{\tilde R}^{n})}^p\lesssim K^{n-1}\|T^{\tilde \lambda}_{\tilde K} \tilde f\|_{L^p(\bar{B}_{\tilde R}^{n})}^p.
	\eeqq
	From $\eqref{eq:123}$, it follows that
	\beqq
	\|T^{\tilde \lambda}_{\tilde K} \tilde f\|_{L^p(\bar{B}_{\tilde R}^{n})}^p\leq C_\varepsilon R^\varepsilon \|\tilde f\|_{L^p(B^{n-1}_1(0))}.
	\eeqq
	By choosing $\delta=\varepsilon^2\ll 1$, we will obtain the desired result $\eqref{re00}$.
	
	Let $1\leq R\leq \lambda$ and  $Q_p(\lambda,R)$ be the optimal constant such that
	\beq
	\|T^\lambda_K  f\|_{L^p(B_R^n)}\leq Q_p(\lambda,R)\|f\|_{L^p(B^n_1(0))}
	\eeq
	holds for all asymptotically flat phase $\phi_K$ in the Definition \ref{def00} and for all $\mathfrak a$ satisfying \eqref{eq:c13}, and uniformly for all $f\in L^p(B^{n-1}_1(0))$.
	Then \eqref{eq:123} is reduced to showing:
	\beq\label{eq:124}
	Q_p(\lambda,R)\leq C_\varepsilon R^\varepsilon.
	\eeq
 We proceed to prove \eqref{eq:124} via an induction on scale argument. For this purpose, we first set up some basic preparatory tools.

	\subsection{Parabolic rescaling and flat decoupling}
	In this section we establish the parabolic rescaling lemma  which connects the estimates at different scales and plays a critical role in the induction argument. To that end, we first prove an auxiliary proposition.

	\begin{proposition}\label{sds}
		Let $\mathcal{D}$ be a maximal $R^{-1}$-separated discrete subset of $\Omega\subset B_1^{n-1}(0)$, then
		\beq\label{sds1}
		\Big\|\sum_{\xi_\theta\in\mathcal{D}}e^{2\pi i\phi^\lambda_K(\cdot,\xi_\theta)}F(\xi_\theta)\Big\|_{L^p(B_R^n(0))}\lesssim Q_p(\lambda, R)R^{\frac{n-1}{p^\prime}} \|F\|_{\ell^p(\mathcal{D})}\eeq for all $F:\mathcal{D}\rightarrow\mathbb{C}$, where
		\beqq
		\|F\|_{\ell^p(\mathcal{D})}:=\Big(\sum_{\xi_\theta \in \mathcal{D}}|F(\xi_\theta)|^p\Big)^{\frac{1}{p}},
		\eeqq
		for $1\leq p<\infty$. \end{proposition}
	\begin{proof}
		Let $\eta$ be a bump smooth function on $\mathbb{R}^{n-1}$, which is  supported on $B_2^{n-1}(0)$ and equals to $1$ on $B_1^{n-1}(0)$. For each $\xi_\theta\in\mathcal{D}$, we set $\eta_{\theta}(\xi):=\eta(10R(\xi-\xi_\theta))$. 
		In exactly the same way as in the proof of Lemma 11.8 of \cite{GHI}, we have
		\beq\label{sds2}
		\Big| \sum_{\xi_\theta\in\mathcal{D}}e^{2\pi i\phi^\lambda_K(\cdot,\xi_\theta)}F(\xi_\theta)\Big|\lesssim R^{n-1}\sum_{k\in\mathbb{Z}^n}(1+|k|)^{-(n+1)}|T^\lambda_K f_k(x)|,\eeq
		where
		\beqq f_k(\xi ):=\sum_{\xi_\theta\in\mathcal{D}}F(\xi_\theta)c_{k,\theta}(\xi)\eta_\theta(\xi)\eeqq
		with  $\|c_{k,\theta}(\xi)\|_\infty\leq1$. By the definition of $Q_p(\lambda, R)$ and \eqref{sds2}, we get
		\beqq
		\Big\|\sum_{\xi_\theta\in\mathcal{D}}e^{2\pi i\phi^\lambda_K(\cdot,\xi_\theta)}F(\xi_\theta)\Big\|_{L^p(B_R^n(0))}\lesssim Q_p(\lambda, R)R^{n-1}\sum_{k\in\mathbb{Z}^n}(1+|k|)^{-(n+1)}\|f_k\|_{L^p(B_1^{n-1}(0))}.\eeqq
		The supports of $\{\eta_\theta\}$ are pairwise disjoint, for any $q\geq1$, we have
		\beqq
		\|f_k\|_{L^q(B^{n-1}_2(0))}\lesssim R^{-\frac{n-1}q} \|F\|_{\ell^q(\mathcal{D})}.\eeqq
		Thus we get
		\begin{align*}
			\Big\|\sum_{\xi_\theta\in\mathcal{D}}e^{2\pi i\phi^\lambda_K(\cdot,\xi_\theta)}F(\xi_\theta)\Big\|_{L^p(B_R^n(0))}&\lesssim Q_p(\lambda,R)R^{n-1}\sum_{k\in\mathbb{Z}^n}(1+|k|)^{-(n+1)}R^{-\frac{n-1}p} \|F\|_{\ell^p(\mathcal{D})}\\
			&\lesssim Q_p(\lambda, R)R^{\frac{n-1}{p^\prime} }\|F\|_{\ell^p(\mathcal{D})}.
		\end{align*}
	\end{proof}
	
	\begin{lemma}{\label{rescaling}}{(Parabolic rescaling)}
		Let $1\leqslant R\leqslant \lambda$, and $f$ be supported in  a ball of radius $K^{-1}$, where $1\leqslant K\leqslant R$. Then for all $p\geqslant2$ and $\delta>0$, we have
		\beq
		\|T^\lambda_K f\|_{L^p(B_R^n(0))}\lesssim_\delta Q_p\Big(\frac{\lambda}{K^2},\frac{R}{K^2}\Big)R^\delta K^{\frac{2n}p-(n-1)}\|f\|_{L^p(B_1^{n-1}(0))}.
		\eeq
	\end{lemma}
	\begin{proof}
		Without loss of generality, we may assume ${\rm supp\,}f\subset B_{K^{-1}}^{n-1}(\bar \xi)$.
		In the same argument as in Section \ref{pre}, we obtain
		\beqq
		\|T^\lambda_Kf\|_{L^p(B_R^n(0))}\lesssim_\delta K^{\frac{n+1}p}\|\widetilde{T}^{\tilde \lambda}_{\tilde K}\tilde{f}\|_{L^p(\Box_R)},
		\eeqq
		where  $\Box_R$ and $\tilde f$ are  defined in the Lemma \ref{la11} and
\begin{align}\tilde \lambda=K^{-2}\lambda,\quad \tilde K=K^{1-2\varepsilon^2}.\end{align}
 Note that for $q\geq 1$,
		\beqq
		\|\tilde{f}\|_{L^q(B^{n-1}_1(0))}\leq K^{-(n-1)+(n-1)/q}\|f\|_{L^q(B^{n-1}_1(0))},\eeqq
		it suffices to show that
		\beqq \|\widetilde{T}^{\tilde  \lambda}_{\tilde K}\tilde{f}\|_{L^p(\Box_R)}\lesssim_\delta Q_p\Big(\tilde \lambda,\tilde R\Big)R^\delta \|\tilde{f}\|_{L^p(B^{n-1}_1(0))}.\eeqq
		To simplify  notations, we just need to show
		\beq \|T^{\lambda}_Kf\|_{L^p(\Box(R,R'))}\lesssim_\delta Q_p(\lambda, R)R^\delta \|f\|_{L^p(B^{n-1}_1(0))}\eeq
		for all $1\ll R\leq R^\prime\leq \lambda$ and $\delta>0$, where
		$$\Box(R,R'):=\left\{x=(x',x_n)\in\mathbb{R}^n:\left(\frac{|x^\prime|}{R^\prime}\right)^2+\left(\frac{|x_n|}{R}\right)^2\leq1\right\}.$$
		Choosing a collection of essentially disjoint $R^{-1}$-balls $\theta$ which covers $B^{n-1}_1(0)$, we denote the center of $\theta$ by $\xi_\theta$ and  decompose $f$ into $f=\sum_\theta f_\theta$. Set
		\beqq
		T^\lambda_{K,\theta}f(x):=e^{-2\pi i\phi^\lambda_K(x,\xi_\theta)}T^\lambda_K f(x),\eeqq
		and we rewrite
		\beqq
		T^\lambda_K f(x)=\sum_\theta e^{2\pi i\phi^\lambda_K(x,\xi_\theta)}T^\lambda_{K,\theta} f_\theta(x).\eeqq
		For sufficiently small $\delta>0$, we may also write
		\beq
		T^\lambda_{K,\theta} f_\theta(x)=T^\lambda_{K,\theta} f_\theta*\eta_{R^{1-\delta}}(x)+{\rm RapDec}(R)\|f\|_{L^2(B^{n-1})},\eeq
		where $\eta$ is a Schwartz function on $\R^n$ and has Fourier support on $B^n_{2}(0)$, and $\hat\eta=1$ on $B^{n}_1(0)$. Then
 $|\eta|$ admits a smooth rapidly decreasing majorant $\zeta:\mathbb{R}^n\rightarrow[0,+\infty)$ which satisfies
		\beq\label{zeta}
		\zeta_{R^{1-\delta}}(x)\lesssim R^\delta \zeta_{R^{1-\delta}}(y) \quad \text{if}\quad |x-y|\lesssim R.\eeq
		Cover $\Box(R,R')$ by a finitely-overlapping $R$-balls $\{B_R^n\}$. For any $B_R^n(\bar x)$ in this cover and for $z\in B_R^n(0)$, we have
		\beqq
		|T^\lambda_K f(\bar{x}+z)|\lesssim R^\delta\int_{\mathbb{R}^n}\Big|\sum_{\theta} e^{2\pi i\phi^\lambda_K(\bar{x}+z,\xi_\theta)}T^\lambda_{K,\theta} f_\theta(y)\Big| \zeta_{R^{1-\delta}}(\bar{x}-y)dy.\eeqq
		Taking the $L^p$-norm in $z$ and using Proposition \ref{sds} for the phase $\phi^\lambda_K(\bar{x}+\cdot,\xi_\theta)$, we have
		\begin{align*}
			\|T^\lambda_K f(\bar x+\cdot)\|_{L^p(B_R^n(0))}& \lesssim  R^\delta\int_{\mathbb{R}^n}\Big\|\sum_{\theta} e^{2\pi i\phi^\lambda_K(\bar{x}+z,\xi_\theta)}T^\lambda_{K,\theta} f_\theta(y)\Big\|_{L^p(B_R^n(0))} \zeta_{R^{1-\delta}}(\bar{x}-y)dy \\
			& \lesssim Q_p(\lambda,R) R^{\frac{n-1}{p^\prime}} R^\delta\int_{\mathbb{R}^n}\|T^\lambda_{K,\theta} f_\theta(y)\|_{\ell^p(\theta)}\zeta_{R^{1-\delta}}(\bar{x}-y)dy,
		\end{align*}
		where  $\|a_{\theta}\|_{\ell^p(\theta)}$ is denoted by  $\big(\sum\limits_{\theta}|a_\theta|^p\big)^{1/p}$.
		
        By the property \eqref{zeta}, for $z\in B_R^n(0)$, we obtain
		\begin{align*}
			&\int_{\mathbb{R}^n}\|T^\lambda_{K,\theta} f_\theta(y)\|_{\ell^p(\theta)}\zeta_{R^{1-\delta}}(\bar{x}-y)dy\\
			=&\int_{\mathbb{R}^n}\|T^\lambda_{K,\theta} f_\theta(\bar{x}+z-y)\|_{\ell^p(\theta)}\zeta_{R^{1-\delta}}(y-z)dy \\
			\lesssim &R^{\delta}\int_{\mathbb{R}^n}\|T^\lambda_{K,\theta} f_\theta(\bar{x}+z-y)\|_{\ell^p(\theta)}\zeta_{R^{1-\delta}}(y)dy\\
			\lesssim& R^{O(\delta)}\left(\int_{\mathbb{R}^n}\|T^\lambda_{K,\theta} f_\theta(\bar{x}+z-y)\|_{\ell^p(\theta)}^{p}\zeta_{R^{1-\delta}}(y)dy\right)^{1/p}.
		\end{align*}
		Then we deduces that for all $z\in B_R^n(0)$
		\beqq
		\begin{aligned}
			\quad \|T^\lambda_K f(\bar x+\cdot)\|_{L^p(B_R^n(0))}&\lesssim Q_p(\lambda, R) R^{\frac{n-1}{p^\prime}}  R^{O(\delta)}\\ &\times \left(\int_{\mathbb{R}^n}\|T^\lambda_{K,\theta} f_\theta(\bar{x}+z-y)\|_{\ell^p(\theta)}^{p}\zeta_{R^{1-\delta}}(y)dy\right)^{1/p}.
		\end{aligned}\eeqq
		Raising both sides of this estimate to the $p$th power, averaging in $z$ and summing over all balls $B_R^n$ in the covering, we conclude
		that $\|T^\lambda_K f\|_{L^p(\Box(R,R'))}$ is dominated by
		\beqq
		Q_p(\lambda,R) R^{\frac{n-1}{p^\prime}-\frac np} R^{O(\delta)}\left(\int_{\mathbb{R}^n}\sum_\theta\|T^\lambda_{K,\theta} f_\theta\|_{L^p(\Box(R,R')-y)}^{p}\zeta_{R^{1-\delta}}(y)dy\right)^{1/p}.\eeqq
		Using the trivial estimate
		\beq
		\|T^\lambda_{K,\theta} f_\theta\|_{L^\infty(\Box(R,R')-y)}\lesssim\|f_\theta\|_{L^1(B_1^{n-1}(0))}\lesssim R^{-(n-1)}\|f_\theta\|_{L^\infty(B_1^{n-1}(0))}\eeq
		and
		\beq
		\|T^\lambda_{K,\theta} f_\theta\|_{L^2(\Box(R,R')-y)}\lesssim R^{1/2}\|f_\theta\|_{L^2(B_1^{n-1}(0))},\eeq
we have
\[\|T_{K,\theta}^\lambda f_\theta\|_{L^p(\Box(R,R')-y)}\lesssim R^{-(2n-1)(\frac12-\frac1p)+\frac12}\|f_\theta\|_{L^p(B_1^{n-1}(0))}.\]
		Hence  $\|T^\lambda_Kf\|_{L^p(\Box(R,R'))}$ is dominated by $  Q_p(\lambda,R) R^{O(\delta)}\|f\|_{L^p(B_1^{n-1}(0))}.$
	\end{proof}

\begin{lemma}\label{lem6666}
Suppose that ${\rm supp\,}f\subset B^{n-1}_1(0)$.
Then the Fourier transform of $T_K^\lambda f$ is essentially supported on the $K^{-2}$-neighborhood of the surface $S:=\{(\xi,\langle M\xi,\xi\rangle):\xi\in \mathrm{supp\,}f\}$ in the sense that
\begin{align}
|\widehat{T^\lambda_K f}(\omega)|\leq{\rm RapDec}(\lambda)\|f\|_{L^p(B_1^{n-1}(0))},\quad \text{for all \,}\omega\notin \mathcal N_{CK^{-2}}S.
\end{align}
\end{lemma}
\begin{proof}
Define
\[G_\lambda(\xi,\omega):=\int_{\R^n}e^{2\pi i(\phi_K^\lambda(x,\xi)-x\cdot\omega)}\mathfrak a^{\lambda }(x,\xi)dx.\]
Then we have
\begin{align}\label{lem7777}\widehat {T^\lambda_Kf}(\omega)=\int_{\R^n}e^{-2\pi ix\cdot\omega}T^\lambda_Kf(x)dx=\int_{B_1^{n-1}(0)}f(\xi)G_\lambda(\xi,\omega)d\xi.\end{align}
We rewrite  as
\[G_\lambda(\xi,\omega)=\lambda^n \int_{\R^n}e^{2\pi i\lambda(\phi_K(y,\xi)-y\cdot\omega)}\mathfrak a(y,\xi)dy.\]
 From integration by parts  and the assumption \eqref{eq:c13} of $\mathfrak a$,   it follows that
\begin{align}\label{lem8888}|G_\lambda(\xi,\omega)|\leq {\rm RapDec}(\lambda),\end{align}
provided that
\begin{align}\label{lem9999}
|\omega-\nabla_y\phi_K(y,\xi)|\geq CK^{-2}.
\end{align}
Since $\phi_K$ is asymptotically flat, $\eqref{lem9999}$ holds true obviously if $\omega\notin \mathcal N_{CK^{-2}}S$.
Combining $\eqref{lem7777}$ and $\eqref{lem8888}$, we have
\[|\widehat{T^\lambda_K f}(\omega)|\leq{\rm RapDec}(\lambda)\|f\|_{L^p(B_1^{n-1}(0))},\quad \text{for all \,}\omega\notin \mathcal N_{CK^{-2}}S.\]
\end{proof}

To prove $\eqref{eq:124}$, we also need a flat decoupling estimate for $T^\lambda_K$.
	\begin{lemma}\label{flat-decoupling}(Flat decoupling)
		Let $\{\tau\}$ be a collection of finitely-overlapping $K^{-1}$-balls covering $B^{n-1}_1(0)$ with $1\leq K \leq R$,
		then we can decompose $f$ as
\[f=\sum_\tau f_\tau.\]
For $2\leq p\leq\infty$, one has
		\begin{align}\label {dec0000}\|T^\lambda_K f\|_{L^p(B_R)}\lesssim (\#\{\tau\})^{\frac12-\frac1p}\Big(\sum_{\tau}\|T^\lambda_Kf_\tau\|_{L^p(\omega_{B_R})}^2\Big)^\frac12+{\rm RapDec}(\lambda)\|f\|_{L^2(B_1^{n-1}(0))}.\end{align}
	\end{lemma}
	\begin{proof}For $p=\infty$, the estimate $\eqref{dec0000}$ is trivial by H\"older's inequality. By interpolation we just need to show $\eqref {dec0000}$ for $p=2$.
		Using Lemma \ref{lem6666} for each $f_\tau$, we get
		\begin{align}\label{eta00}
			T_K^\lambda f_{\tau}={\chi}_{\mathcal N_{CK^{-2}}(S_\tau)}(D)T_K^\lambda f_\tau+{\rm RapDec}(\lambda)\|f\|_{L^2(B_1^{n-1}(0))},
		\end{align}
		where $S_\tau:=\{(\xi, \langle M\xi,\xi): \xi\in \tau\}$.
 Note that the $CK^{-2}$-neighborhoods of $S_\tau$ are finitely overlappling,
 then we complete the proof by making use of Plancherel's theorem.
	\end{proof}
	\subsection{Bilinear restriction estimate}
	Assume that $\phi$ satisfies the Carleson-Sj\"olin conditions.
	Let $U_1,U_2$ be two balls contained in $B_1^{n-1}(0)$ and $\xi_i\in U_i,i=1,2$.
	By the assumption (H1), the map
	\begin{equation*}
		\xi \mapsto \partial_{x'}\phi(x,\cdot)
	\end{equation*}
	is a diffeomorphism. Define
	\begin{equation*}
		q(x,\xi):=\partial_{x_n}\phi(x,(\partial_{x'}\phi(x,\cdot))^{-1}(\xi)),
	\end{equation*}
	i.e.
	\begin{equation}
		q(x,\partial_{x'}\phi(x,\xi))=\partial_{x_n}\phi(x,\xi).
	\end{equation}
	\begin{theorem}\cite{LS} \label{theo3}
		Let $\phi(x,\xi_i),i=1,2$ satisfy the conditions ${\rm (H1), (H2)}$. Assume that $(x,\xi_i)\in {\rm supp}\;a_i$, if $\partial_{\xi\xi}^2q$ satisfies
		\beqq
		{\rm det} \partial_{\xi\xi}^2 q(x,\partial_{x'}\phi(x,\xi_i))\neq 0, \quad \text{if}\; \xi_i\in {\rm supp}\;a_i(x,\cdot),
		\eeqq
		and
		\beq\label{eq:4}
		\big|\langle \partial_{x'\xi}^2\phi(x,\xi)\delta(x,\xi_1,\xi_2), [\partial_{x'\xi}^2\phi(x,\xi_i)]^{-1}[\partial_{\xi\xi}^2 q(x,u_i)]^{-1}\delta(x,\xi_1,\xi_2)\rangle \big|\geq c>0,
		\eeq
		for $i=1,2$, where $u_i=\partial_{x'}\phi(x,\xi_i)$ and $\delta(x,\xi_1,\xi_2)=\partial_{\xi}q(x,u_1)-\partial_\xi q(x,u_2)$,
		then
		\beq\label{eq:1}
		\big\||T^\lambda f_1T^\lambda f_2|^\frac12\big\|_{L^p(B_R^n)}\lesssim_{\phi,\varepsilon}R^\varepsilon \prod_{i=1}^2 \|f_i\|_{L^2}^\frac12,
		\eeq
		for $p \geq \frac{2(n+2)}{n}$.
	\end{theorem}	



	To apply Theorem \ref{theo3} to study the oscillatory operator $T_K^\lambda$,   we first introduce a notion of  \emph{strongly separated condition}.
	\begin{definition}[Strongly separated condition]
		Let $\tau_1,\tau_2$ be two balls of of dimension  $K^{-1}$. We say $\tau_1,\tau_2$ satisfy the strongly separated condition  if for each $\xi_i \in \tau_i$, the condition
		\beq \label{strong separated}
		\big|\langle \partial_{x'\xi}^2\phi(x,\xi)\delta(x,\xi_1,\xi_2), [\partial_{x'\xi}^2\phi(x,\xi_i)]^{-1}[\partial_{\xi\xi}^2 q(x,u_i)]^{-1}\delta(x,\xi_1,\xi_2)\rangle \big|\geq CK^{-1}
		\eeq holds.
	\end{definition}
	The next proposition concerns a geometric lemma  associated with the phase $\phi_K^\lambda$.
	\begin{proposition}\label{pro1}
		Let $\{\tau\}$ be a family of finitely-overlapping balls of radius $K^{-1}$. Then  we have the following two dichotomies:
		
		{\rm(I)} There exists  an $m$-dimensional affine subspace $V$ such that  every $\tau$ is contained in an $O(K^{-\frac{1}{2n}})$ neighbourhood of $V$.
		
		{\rm(II)} There are two $K^{-1}$-balls $\tau,\tau'$ which satisfy the strongly separated condition associated with $\phi_K$.
	\end{proposition}
	 Barron \cite{Bar} proved the above proposition for the standard phase  $\phi(x,\xi)=x'\cdot \xi+x_n\langle M\xi,\xi\rangle$. Note that
\[\phi_{K}(x,\xi)=x'\cdot\xi+x_n\langle M\xi,\xi\rangle+{\rm E}_K(x,\xi)\]
 can be viewed as a small perturbation of the standard case, and the perturbation is  sufficiently small comparing to $K^{-1}$, thus the strongly separated condition under the phase $\phi_K$ can be essentially  identified the same as the standard phase $x\cdot \xi+\langle M\xi,\xi\rangle $.
	\subsection{Broad-Narrow analysis.}
	Let $\delta=\varepsilon^2\ll1$, and set
	\begin{align}\label{relation}
		K_2=K_1^{2\delta},\, K_1=K^\alpha,\, K=K_0R^\delta,
	\end{align}
	where $\alpha=\frac1{2n}$.
	Let $\mathfrak T$ be a collection of finitely-overlapping $K^{-1}$-balls $\tau$ covering $\mathrm{supp\,} f$, and we can fix a collection $\mathscr Q$ of finitely-overlapping $K^2$-cubes that cover $B^{n}_R(0)$.  For each $Q\in\mathscr Q$, we define its \emph{significant set}
	\[\mathcal S_p(Q):=\Big\{\tau\in\mathfrak T:\|T^\lambda_K f_{\tau}\|_{L^p(Q)}\geq\frac1{100\#\mathfrak T}\|T^\lambda_K f\|_{L^p(Q)}\Big\}.\]
	We say that a $K^2$-cube $Q\in\mathscr Q$ is \emph{narrow} and write $Q\in\mathscr N$
	if and only if there exists an $m$-dimensional linear subspace $V\subset\R^{n}$
	such that
	\begin{align}\angle(G^\lambda(x,\tau),V)\leq CK^{-1}_1\end{align}
	for all $\tau\in \mathcal S_p(Q)$, here  for given $x\in B_R^n(0)$,  $G^\lambda(x,\tau)$ denotes the set of the unit normal vectors of the hypersurface $\{\partial_x\phi^\lambda_K(x ,\eta):\eta \in\tau\}$.
	If a $K^2$-cube $Q\in\mathscr Q$ is not narrow,  then we call it  \emph{broad} and write $Q\in\mathscr B$. Thus
	\[\|T^\lambda_K f\|_{L^p(B_R^n )}^p\leq \sum_{Q\in\mathscr B}\|T^\lambda_K f\|_{L^p(Q)}^p+\sum_{Q\in\mathscr N}\|T^\lambda_K f\|_{L^p(Q)}^p.
	\]
	We call it broad case if
	\[\|T^\lambda_K f\|_{L^p(B_R)}^p\leq 2\sum_{Q\in\mathscr B}\|T^\lambda_K f\|_{L^p(Q)}^p,
	\]
	otherwise narrow case if
	\[\|T^\lambda_K f\|_{L^p(B_R)}^p\leq 2\sum_{Q\in\mathscr N}\|T^\lambda _Kf\|_{L^p(Q)}^p.\]
	Now we are going to prove \eqref{eq:124}. Obviously \eqref{eq:124} holds true for $1\leq\lambda\leq 1000$, so let us suppose \eqref{eq:124} holds true for $1\leq r\leq\lambda'\leq\lambda/2$.
	In the following part, we will deal with the broad  and narrow cases respectively. Then we balance the two cases and close  the whole induction for \eqref{eq:124}.

	\subsection{Narrow estimate}
	Suppose $Q\in\mathscr Q$ is a narrow cube, by Proposition \ref{pro1},
	there exists an $m$-dimensional affine subspace $V\subset \R^{n-1}$ such that
	\[\bigcup_{\tau\in S_p(Q)}\tau\subset N_{CK_1^{-1}}V.\]
	We decompose $B^{n-1}_1(0)$ into $K_1^{-1}$-balls $\{\pi\}$. Let $\Pi_V$ be a minimal collection of $\{\pi\}$ covering $B^{n-1}_1(0)\cap N_{CK_1^{-1}}V$ and $\mathfrak I$ be a collection of finitely-overlapping $K_1^{-1}$-balls $\{\pi\}$ covering $\mathrm{supp\,} f$. Note that $\Pi_V$ contains $CK_1^m$ many balls $\pi$. Using  H\"older's inequality and Lemma \ref{flat-decoupling}, we obtain
	\begin{align*}
		\|T^\lambda_K f\|_{L^p(Q)}
		&\leq CK_1^{m(\frac12-\frac1p)}\Big(\sum_{\pi \in\Pi_V}\|T_K^\lambda f_\pi\|_{L^p(\omega_Q)}^2\Big)^\frac12\\
		&\leq CK_1^{2m(\frac12-\frac1p)}\Big(\sum_{\pi \in\Pi_V}\|T^\lambda_K f_\pi \|_{L^p(\omega_Q)}^p\Big)^\frac1p\\
		&\leq CK_1^{2m(\frac12-\frac1p)}\Big(\sum_{\pi \in\mathfrak I}\|T^\lambda_K f_\pi \|_{L^p(\omega_Q)}^p\Big)^\frac1p.
	\end{align*}
	By  Lemma \ref{rescaling} and our induction assumption, we have
	\begin{align*}
		\Big(\sum_{Q\in\mathscr N}\|T_K^\lambda f\|_{L^p(Q)}^p\Big)^\frac1p
		&\leq CK_1^{2m(\frac12-\frac1p)}\Big(\sum_{\pi\in\mathfrak I }\|T_K^\lambda f_\pi \|_{L^p(\omega_{B_R})}^p\Big)^\frac1p\\
		&\leq \bar{C}C_\varepsilon R^\varepsilon K_1^{-\varepsilon}K_1^{2m(\frac12-\frac1p)-(n-1)+\frac{2n}p}\Big(\sum_{\pi\in\mathfrak I }\|f_\pi \|_{L^p(B_1^{n-1}(0))}^p\Big)^\frac1p\\
		&\leq \bar{C}C_\varepsilon R^\varepsilon K_1^{-\varepsilon}K_1^{2m(\frac12-\frac1p)-(n-1)+\frac{2n}p}\|f\|_{L^p(B_1^{n-1}(0))},
	\end{align*}
where $\bar{C}$ is a large constant.

	 If  $p\geq\frac{2(n-m)}{n-m-1}$, we obtain
	\beq\label{narrow result}
	\Big(\sum_{Q\in\mathscr N}\|T_K^\lambda f\|_{L^p(Q)}^p\Big)^\frac1p
	\leq C_\varepsilon R^\varepsilon \|f\|_{L^p(B_1^{n-1}(0))}.
	\eeq
	\subsection{Broad estimate}
	We will show the broad estimate by making use of the bilinear arguments.
	
	\begin{proposition}[Broad estimate]\label{broad result}
		Let $p\geq\frac{2(n+2)}{n}$. We have
		\begin{align}\label{broad result0}
			\sum_{Q\in\mathscr B}\|T_K^\lambda f\|_{L^p(Q)}^p\leq CK^{O(1)}\|f\|_{L^2(B^{n-1}_1(0))}^p.
		\end{align}
	\end{proposition}
	To prove Proposition \ref{broad result}, we naturally need to obtain  the bounds of $\|T^\lambda_Kf\|_{L^p(Q)}^p$ for each $Q$ firstly, and then sum them together.
	For this purpose we first present two lemmas.
	\begin{lemma}\label{lem:broad-1}
		For any $Q\in \mathscr B$, there are two $K^{-1}$-balls $\tau_1,\tau_2\in \mathcal S_p(Q)$ satisfying the strongly separated condition \eqref{strong separated} such that
		\begin{align}\|T_K^\lambda f\|_{L^p(Q)}\leq CK^{O(1)}\|T_K^\lambda f_{\tau_1}\|_{L^p(Q)}^\frac12\|T_K^\lambda f_{\tau_2}\|_{L^p(Q)}^\frac12.\end{align}
	\end{lemma}

	\begin{proof} Let $Q\in\mathscr B$, then $\#\mathcal S_p(Q)\geq2$. Suppose that there doesn't exist two $K^{-1}$-balls $ \tau_1, \tau_2\in \mathcal S_p(Q)$ satisfying the strongly separated condition \eqref{strong separated}.
		Applying the Proposition \ref{pro1} to $\mathcal S_p(Q)$, we get
		\[ \tau\subset N_{CK_1^{-1}}V,\quad \text{for all\ }  \tau\in \mathcal S_p(Q)\]
		for some $m$-dimensional affine subspace $V$. This force all $G^\lambda(x,\tau)$ to be in the neighborhood $\mathcal N_{CK_1^{-1}}W$ of some $m$ dimensional  subspace $W$. Thus
		$Q$ is a narrow cube, which contradicts with our assumption.  Thus, we can find $ \tau_1, \tau_2\in\mathcal S_p(Q)$ satisfying the strongly separated condition \eqref{strong separated} such that
		\begin{align}
			\|T_K^\lambda f\|_{L^p(Q)}
			&\leq(100\#\mathfrak T)\|T_K^\lambda f_{\tau_1}\|_{L^p(Q)}^\frac12\|T_K^\lambda f_{\tau_2}\|_{L^p(Q)}^\frac12 \nonumber\\
			&\leq CK^{O(1)}\|T_K^\lambda f_{\tau_1}\|_{L^p(Q)}^\frac12\|T_K^\lambda f_{\tau_2}\|_{L^p(Q)}^\frac12.
		\end{align}
	\end{proof}

\begin{lemma}\label{lemmaoo}
Suppose that $f\in L^2(\R^{n-1})$ with support ${\rm supp\,}f\subset B^{n-1}_{K^{-1}}(\bar\xi)\subset B^{n-1}_1(0)$, then we have
\begin{align}\label{lemmaoooo}
|T_K^\lambda f(x)|=|(e^{-2\pi i\phi_K^\lambda(\cdot,\bar\xi)}T_K^\lambda f)\ast \psi_{K/C}(x)|+{\rm RapDec}(\lambda)\|f\|_{L^2(\R^{n-1})}
\end{align}
where $\psi_{K/C}(x)=C^nK^{-n}\psi(CK^{-1}x)$ with ${\rm supp\,}\hat\psi \subset B_2^n(0)$ and $\hat\psi=1$ on $B_1^n(0)$.
\end{lemma}

\begin{proof}We observe that
\begin{align*}
\mathcal F\big(e^{-2\pi i\phi_K^\lambda(\cdot,\bar\xi)}T^\lambda_Kf(\cdot)\big)(\omega)
&=\int_{\R^n}\int_{B_1^{n-1}(0)}e^{2\pi i(\phi_K^\lambda(x,\xi)-\phi_K^\lambda(x,\bar\xi)-x\cdot\omega)}\mathfrak a^\lambda(x,\xi)f(\xi)d\xi dx \\
&=\int_{B_1^{n-1}(0)}F_\lambda(\xi,\omega)f(\xi)d\xi,
\end{align*}
where
\[F_\lambda(\xi,\omega):=\int_{\R^n}e^{2\pi i(\phi_K^\lambda(x,\xi)-\phi_K^\lambda(x,\bar\xi)-x\cdot\omega)}\mathfrak a^\lambda(x,\xi)dx.\]
We can rewrite as
\[F_\lambda(\xi,\omega):=\lambda^n\int_{\R^n}e^{2\pi i\lambda(\phi_K(x,\xi)-\phi_K(x,\bar\xi)-x\cdot\omega)}\mathfrak a(x,\xi)dx.\]
For $|\omega|\geq CK^{-1}$ and $x\in B_1^n(0)$, we have $$|\nabla_x\phi_K(x,\xi)-\nabla_x\phi_K(x,\bar\xi)-\omega|\geq K^{-1}.$$ Integration by parts, it follows  that
\begin{align}|F_\lambda(\xi,\omega)|\leq {\rm RapDec}(\lambda)(1+|\omega|)^{-(n+1)}. \end{align}
Thus we have
\[\mathcal F\Big(e^{-2\pi i\phi_K^\lambda(\cdot,\bar\xi)}T^\lambda_K f(\cdot)\Big)(\omega)=\widehat{\psi_{K/C}}(\omega)\mathcal F\Big(e^{-2\pi i\phi_K^\lambda(\cdot,\bar\xi)}T^\lambda_K f(\cdot)\Big)(\omega)+U(f,\lambda)(\omega),\]
where $|U(f,\lambda)(\omega)|\leq {\rm RapDec}(\lambda)(1+|\omega|)^{-(n+1)}\|f\|_{L^2(B_1^{n-1}(0))}$. By making use of Fourier inversion we obtain
 \[e^{-2\pi i\phi_K^\lambda(x,\bar\xi)}T^\lambda_Kf(x)=\psi_{K/C}\ast\big(e^{-2\pi i\phi_K^\lambda(\cdot,\bar\xi)}T^\lambda_K f(\cdot)\big)(x)+{\rm Rap Dec}(\lambda)\|f\|_{L^2(B_1^{n-1}(0))}.\]
Then \eqref{lemmaoooo} holds true obviously.
\end{proof}

	\begin{lemma}\label{lem:broad-2}
		For any two $K^{-1} $-balls $ \tau_1, \tau_2$ satisfying  the strongly separated condition \eqref{strong separated}, there holds
		\begin{align}\label{bilinear type}
			\sum_{Q\in\mathscr B}\|T_K^\lambda f_{\tau_1}\|_{L^p(Q)}^\frac p2\|T_K^\lambda f_{\tau_2}\|_{L^p(Q)}^\frac p2\leq K^{O(1)}\|f\|_{L^2(B_1^{n-1}(0))}^p+{\rm RapDec}(\lambda)\|f\|_{L^p(B_1^{n-1}(0))}.
		\end{align}
	\end{lemma}
	\begin{proof}Without loss of generality we may assume $\|f\|_{L^2(B^{n-1}_1(0))}=1$. By Lemma \ref{lemmaoo}, we have
\[|T^\lambda_Kf_\tau(x)|=\big|\big(e^{-2\pi i\phi_K^\lambda(\cdot,\bar\xi)}T^\lambda_Kf_\tau(\cdot)\big)\ast\psi_{K/C}(x)\big|+{\rm RapDec}(\lambda),\]
for each $\tau$.
	To prove $\eqref{bilinear type}$, we just need to show
		\begin{align}
&\sum_{Q\in\mathscr B}\big\|(e^{-2\pi i\phi^\lambda(\cdot,\xi_{\tau_1})}T_K^\lambda  f_{\tau_1})\ast\psi_{K/C}\big\|_{L^\infty(Q)}^\frac p2
\big\|(e^{-2\pi i\phi^\lambda(\cdot,\xi_{\tau_2})}T_K^\lambda  f_{\tau_2})\ast\psi_{K/C}\big\|_{L^\infty(Q)}^\frac p2 \nonumber \\
\leq& K^{O(1)}\|f\|_{L^2(B_1^{n-1}(0))}^p.
\end{align}
		Define $\displaystyle\zeta_K(x):=\sup_{|y-x|\leq K^2}|\psi_{K/C}(x)|$. By the locally constant property one can choose some cube $I_Q\subset Q$ with $|I_Q|\lesssim 1$ such that
		\begin{align*}
			&\big\|(e^{-2\pi i\phi^\lambda(\cdot,\xi_{\tau_1})}
			T_K^\lambda  f_{\tau_1})\ast\psi_{K/C}\big\|_{L^\infty(Q)}\big\|(e^{-2\pi i\phi^\lambda(\cdot,\xi_{\tau_2})}T_K^\lambda  f_{\tau_2})\ast\psi_{K/C}\big\|_{L^\infty(Q)} \\
			\leq& \int_{I_Q}\int_{\mathbb R^n}\int_{\mathbb R^n}|T_K^\lambda f_{\tau_1}(x-y)\zeta_K(y)T_K^\lambda f_{\tau_2}(x-z)\zeta_K(z)|dydzdx.
		\end{align*}
		Then we only need to show
		\[\sum_{Q\in\mathscr B}\Big(\int_{I_Q}\int_{\mathbb R^n}\int_{\mathbb R^n}|T_K^\lambda f_{\tau_1}(x-y)\zeta_K(y)T_K^\lambda f_{\tau_2}(x-z)\zeta_K(z)|dydzdx\Big)^\frac p2\leq K^{O(1)}\|f\|_{L^2(B_1^{n-1}(0))}^p.\]
		Using H\"older's inequality, and for $p\geq \frac{2(n+2)}n$, we have
		\begin{align*}
			&\sum_{Q\in\mathscr B}
			\Big(\int_{I_Q}\int_{\mathbb R^n}\int_{\mathbb R^n}|T_K^\lambda f_{\tau_1}(x-y)\zeta_K(y)T_K^\lambda f_{\tau_2}(x-z)\zeta_K(z)|dydzdx\Big)^\frac p2\\
			\leq &K^{O(1)}\sum_{Q\in\mathscr B}\int_{I_Q}\int_{\mathbb R^n}\int_{\mathbb R^n}|T_K^\lambda f_{\tau_1}(x-y)T_K^\lambda f_{\tau_2}(x-z)|^\frac p2\zeta_K(y)\zeta_K(z) dydzdx\\
			\leq& K^{O(1)}\sup_{y,z}\Big(\int_{B_R^n(0)}|T_K^\lambda f_{\tau_1}(x-y)|^\frac p2|T_K^\lambda f_{\tau_2}(x-z)|^\frac p2dx\Big)\\
			\leq &K^{O(1)}\|f\|_{L^2(B_1^{n-1}(0))}^p.
		\end{align*}
Where we have used Theorem \ref{theo3} in the last inequality. Hence, we complete the proof of Lemma \ref{lem:broad-2}.
	\end{proof}
	Finally, we use the three lemmas above to give a proof of Proposition \ref{broad result}.
	
	{\bf \noindent The proof of broad estimate}: By Lemmas \ref{lem:broad-1}, \ref{lemmaoo} and \ref{lem:broad-2}, we have
	\begin{align*}
		\sum_{Q\in\mathscr B}\|T_K^\lambda f\|_{L^p(Q)}^p
		&\leq CK^{O(1)}\sum_{Q\in\mathscr B}\sum_{\substack{ \tau_1, \tau_2\in \mathcal S_p(Q)\\  \tau_1, \tau_2\ \mathrm{ satisfy } {\eqref{strong separated}}}}\|T_K^\lambda f_{\tau_1}\|_{L^p(Q)}^\frac p2\|T_K^\lambda f_{\tau_2}\|_{L^p(Q)}^\frac p2 \\
		&= CK^{O(1)}\sum_{ \tau_1, \tau_2\ \mathrm{  satisfy } {\eqref{strong separated}}}\sum_{Q\in\mathscr B: \tau_1, \tau_2\in\mathcal S_p(Q)}\|T_K^\lambda f_{\tau_1}\|_{L^p(Q)}^\frac p2\|T_K^\lambda f_{\tau_2}\|_{L^p(Q)}^\frac p2 \\
		&\leq CK^{O(1)}\|f\|_{L^2(B_1^{n-1}(0))}^p,
	\end{align*}
	where we have used the fact that $\mathfrak T\leq K^{O(1)}$ in the last inequality. Then we finished the proof of Proposition \ref{broad result}.
	
	For all $m\leq \lfloor(n-1)/2\rfloor$, we prove Theorem \ref{theo1} by using the narrow estimate \eqref{narrow result} and the broad estimate \eqref{broad result0}.

	Recall \eqref{eq:124} holds true for $1\leq\lambda'\leq \lambda/2$, thus we have $Q_p(\tilde \lambda,\tilde R)\leq C_\varepsilon \tilde R^\varepsilon$, where $\tilde \lambda=K^{-2}\lambda<\lambda/2$ and $\tilde R=K^{-2}R\leq \tilde \lambda$.
 Thanks to the relation of $K,K_1$ and $R$ in \eqref{relation}, and we conclude that
	\[Q_p(\lambda,R)\leq C_\varepsilon R^\varepsilon K_1^{-\varepsilon}+CK^{O(1)}\leq C_\varepsilon R^\varepsilon\]
	holds true for
	\begin{align}\label{p-relation}p\geq\max_{0\leq m\leq \lfloor \frac{n-1}2\rfloor}\Big\{\frac{2(n+2)}n,\frac{2(n-m)}{n-m-1}\Big\}.\end{align}
This inequality  is equivalent to
	\begin{equation}\label{eq:maina1a}p\geq \left\{\begin{aligned}
			&\tfrac{2(n+1)}{n-1}\quad \text{\rm for $n$ odd},\\
			&\tfrac{2(n+2)}{n}\quad \text{\rm for $n$ even}.
		\end{aligned}\right.\end{equation}
	Then we finished the proof of Theorem \ref{theo1}.

	\section{proof of the decoupling theorem}\label{dec}
	\subsection{Reduction}
	First, we recall the decoupling theorem of Bourgain-Demeter\cite{BD1}. Let $S$ be a compact hypersurface with nonvanishing Gaussian curvature, and $\mathcal{N}_{\delta}(S)$ be  the $\delta$-neighborhood of $S$. Decompose $\mathcal{N}_\delta(S)$ into a collection of finitely-overlapping slabs $\{\vartriangle \}$ of dimension $\delta^{1/2}$ in the tangent direction and $\delta$ in the normal direction. We have the following  decomposition
	\beqq
	f=\sum_\vartriangle  f_\vartriangle,
	\eeqq
	where ${\rm supp}\widehat{f_\vartriangle}\subset \vartriangle $. A classical decoupling result associated to this decomposition is as follows:
	\begin{theorem}[\cite{BD1}]\label{theo13}
		Let $S$ be a compact smooth hypersurface in $\R^n$ with nonvanishing Gaussian curvature. If ${\rm supp}\hat{f}\subset \mathcal{N}_\delta(S)$, then for $p\geq \frac{2(n+1)}{n-1}$ and $\varepsilon>0$,
		\beq
		\|f\|_{L^p(\R^n)}\leq_\varepsilon \delta^{\frac{n}{p}-\frac{n-1}{2}-\varepsilon}\Big(\sum_{\vartriangle}\|f_\vartriangle\|_{L^p(\R^n)}^p\Big)^{\frac{1}{p}}.
		\eeq
	\end{theorem}
They also have a local version of decoupling 
\begin{align}\label{loc-dec}
\|f\|_{L^p(B_R^n)}\leq_\varepsilon \delta^{\frac{n}{p}-\frac{n-1}{2}-\varepsilon}\Big(\sum_{\vartriangle }\|f_\vartriangle \|_{L^p(\omega_{B_R^n})}^p\Big)^{\frac{1}{p}}.
\end{align}

	For $1\ll R\leq \lambda$, let $D_p(\lambda,R)$ be the optimal constant such that
	\begin{align}
	\|T^\lambda_K f\|_{L^p(B_R^n)}\leq D_p(\lambda,R)\Big(\sum_{\theta}\|T^\lambda_K f_\theta\|_{L^p(w_{B_R^n})}^p\Big)^{1/p}+{\rm RapDec}(\lambda)\|f\|_{L^p}
	\end{align}
holds true for all asymptotically flat phase $\phi_K$ and for all $\mathfrak a$ satisfying $\eqref{eq:c13}$, and uniformly for all $f\in L^p(B_1^{n-1}(0))$.
	To prove Theorem \ref{theo31},  it suffices to show
	\beq\label{eq:induc}
	D_p(\lambda,R)\leq C_{\varepsilon}R^{\frac{n-1}2-\frac np+\varepsilon}.
	\eeq
	Indeed, by H\"older's inequality, we have 
	\begin{align}
	\|T^\lambda f\|_{L^p(B_R^n)}\leq K^{(n-1)/p'}\big(\sum_{\tau}\|T^\lambda f_\tau\|_{L^p(B_R^n)}^p\big)^{1/p}.
	\end{align}
	For each $\tau$, performing the similar procedure as in the proof of  Lemma \ref{la11}, we have
	\beqq
	\|T^\lambda f_\tau\|_{L^p(B_R^n)}^p \leq CK^{O(1)}\sum_{B_{\tilde R}^n \subset \Box_R}\|T_{\tilde K}^{\tilde \lambda} \tilde{f}\|_{L^p(B_{\tilde R}^n)}^p,
	\eeqq
	where $\tilde f(\cdot)=K^{-(n-1)}f(K^{-1}\cdot+\xi_\tau)$, $\tilde R=R/K^2, \tilde K=K_0\tilde R^{\varepsilon^2}, \tilde \lambda=\lambda/K^2$, and $\Box_R$ is a rectangle of dimensions $R/K\times \cdots \times R/K\times R/K^2$.
	Then by \eqref{eq:induc}, we have
	\begin{align}
	\|T^{\tilde \lambda}_{\tilde K}\tilde f\|_{L^p(B_{\tilde R}^n)}\leq C_{\varepsilon}(\tilde R)^{\frac{n-1}2-\frac np+\varepsilon}\Big(\sum_{\tilde \theta }\|T_{\tilde K}^{\tilde \lambda}\tilde f_{\tilde \theta}\|_{L^p(w_{B_{\tilde R}^n)}}^p\Big)^{1/p}+{\rm RapDec}(\tilde \lambda)\|f\|_{L^p(B_1^{n-1}(0))},
	\end{align}
	where $\tilde \theta$ is a ball of dimension $\tilde{R}^{-1/2}$. By reversing the change of variables, we finally  obtain
	\beqq
	\|T^\lambda f\|_{L^p(B_R^n)}\leq C_\varepsilon R^{\frac{n-1}2-\frac np+\varepsilon} \Big(\sum_{\theta}\|T^\lambda f_\theta\|_{L^p(w_{B_R^n})}^p\Big)^{1/p}+{\rm RapDec}(\lambda)\|f\|_{L^p(B_1^{n-1}(0))}.
	\eeqq
	\subsection{The proof of \eqref{eq:induc}.} Let $ \tau\subset B_1^{n-1}(0)$ be a ball of radius $K^{-1}$. For convenience, define
	$$\mathbb{H}:=\{(\xi, \langle M\xi,\xi\rangle):\xi \in B_1^n(0)\},$$
	and denote  by $\mathbb{H}_{\tau}$ a cap on $\mathbb{H}$,
	\beqq
	\mathbb{H}_{ \tau}:=\{(\xi, \langle M\xi,\xi\rangle):\xi \in \tau\}.
	\eeqq
	If $\omega $ does not belong to a $CK^{-1}$-neighborhood of $\mathbb{H}$, by Lemma \ref{lem6666} we have
	\beqq
	\widehat{T^{\lambda}_K f}(\omega)={\rm RapDec}(\lambda)\|f\|_{L^p(B_1^{n-1}(0))}.
	\eeqq
	Therefore
	\beq\label{decp555}
	T_{K}^{\lambda}f=\chi_{CK^{-1}}(\mathbb{H})(D)T_{K}^{ \lambda} f+{\rm RapDec}(\lambda)\|f\|_{L^p(B_1^{n-1}(0))}.
	\eeq
	Similarly,
	\beq \label{decp666}
	T_{K}^{\lambda} f_{ \tau}=\chi_{CK^{-1}}(\mathbb{H}_{ \tau})(D)T_{K}^{ \lambda} f_{\tau}+{\rm RapDec}(\lambda)\|f\|_{L^p(B_1^{n-1}(0))}.
	\eeq
	 Applying local decoupling inequality \eqref{loc-dec} to the formula $\eqref{decp555}$ and $\eqref{decp666}$, we have
	\beqq
	\big\|T^{ \lambda}_{ K} f\big\|_{L^p(B_{ R}^{n})}\leq C_{\bar \delta} K^{\frac{n-1}2-\frac np+\bar{\delta}} \Big(\sum_{ \tau }\|T_{K}^{\lambda} f_{ \tau}\|_{L^p(\omega_{B_{ R}^n})}^p\Big)^{1/p}+{\rm RapDec}(\lambda)\|f\|_{L^p(B_1^{n-1}(0))},
	\eeqq
	where $\bar{\delta}>0$ is a small constant to be chosen later.
	By a similar way as in  the proof of Lemma \ref{la11}, we have
	\beq\label{eq:1122}
	\|T_{ K}^{\lambda}f_{ \tau}\|_{L^p(\omega_{B_{ R}^n})}^p\leq C(K)\sum_{B_{\tilde  R}^n \subset \Box_{R}}\|\tilde T^{\tilde {\lambda}}_{\tilde  K}\tilde{f}\|_{L^p(B_{\tilde  R}^n)}^p+{\rm RapDec }(\lambda)\|f\|_{L^p(B_1^n(0))},
	\eeq
where $\tilde f(\xi)=K^{-(n-1)}f(\xi_\tau+K^{-1}\xi)$ and $\xi_\tau$ is the center of $\tau$.
	For each $B_{\tilde R}^n$,  by the definition of $D_p(\lambda,R)$, we have
	\beq
	\|\tilde T^{\tilde {\lambda}}_{\tilde  K}\tilde{f}\|_{L^p(B_{\tilde  R}^n)}\leq D_p(\tilde \lambda,\tilde R)\Big(\sum_{ \tilde \theta }\|\tilde T_{\tilde K}^{\tilde \lambda} \tilde f_{ \tilde \theta}\|_{L^p(w_{B_{\tilde R}^n)}}^p\Big)^{1/p}+{\rm RapDec}(\lambda)\|f\|_{L^p(B_1^{n-1}(0))},
	\eeq
	where $\{\tilde \theta\}$ is a collection of finitely-overlapping balls of radius $\tilde R^{-1/2}$.
	By reversing the change of variables,  finally we have
	\beq
	\big\|T^{ \lambda}_{ K} f\big\|_{L^p(B_{ R}^{n})}\leq C_{\bar \delta}K^{\frac{n-1}{2}-\frac{n}{p}+\bar{\delta}}  D_p(\tilde \lambda,\tilde R)\Big(\sum_{ \theta }\|T_{K}^{\lambda} f_{ \theta}\|_{L^p(w_{B_{ R}^n)}}^p\Big)^{1/p}+{\rm RapDec}(\lambda)\|f\|_{L^p(B_1^{n-1}(0))}.
	\eeq
	Recalling the definition of $D_p(\lambda,R)$, we have
	\beq\label{formula99}
	D_p(\lambda,R)\leq C_{\bar \delta} K^{\frac{n-1}{2}-\frac{n}{p}+\bar \delta} D_p(\tilde \lambda,\tilde R).
	\eeq
	This inequality \eqref{formula99} yields,  by the induction hypothesis that
	\beq\label{formula100}
	D_p(\lambda,R)\leq C_\varepsilon R^{\frac{n-1}2-\frac np+\varepsilon} C_{\bar \delta} K^{\bar \delta-2\varepsilon}.
	\eeq
	Choosing  $\bar{\delta}=\varepsilon^2$ and $K_0$ sufficiently large such that
	\beqq
	K_0^{\varepsilon^2-2\varepsilon}C_{\bar \delta}\leq 1,
	\eeqq
	and from \eqref{formula100}, we can complete the induction procedure, i.e.
	\beqq
	D_p(\lambda,R)\leq C_{\varepsilon}R^{\frac{n-1}2-\frac np+\varepsilon}.
	\eeqq

	\subsection*{Acknowledgements}\;   The project was supported by the National Key R\&D Program of China: No. 2022YFA1005700,   C.Gao was supported by NSFC No.12301121, and C. Miao was supported by NSFC No.12371095.

	\bibliographystyle{amsplain}

\end{document}